\documentstyle[amssymb,amsfonts,12pt]{amsart}
\newtheorem{theorem}{Theorem}[section]
\newtheorem{lemma}[theorem]{Lemma}

\newtheorem{proposition}[theorem]{Proposition}

\newtheorem{conjecture}[theorem]{Conjecture}

\theoremstyle{remark}

\def\QSet{\mbox{\rm\kern.24em
\vrule width.03em height1.48ex depth-.051ex \kern-.26em Q}}

\def\R{{\mathbb R}}

\def\N{{\mathbb N}}
\def\C{{\mathbb C}}
\def\Q{{\mathbb Q}}
\def\Z{{\mathbb Z}}

\def\be#1{\begin{equation}\label{#1}}

\def\bas{\begin{align*}}
\def\eas{\end{align*}}
\def\bi{\begin{itemize}}
\def\ei{\end{itemize}}
\newenvironment{proof}{\noindent {\bf Proof} }{\endprf\par}
\def \endprf{\hfill  {\vrule height6pt width6pt depth0pt}\medskip}
\def\emph#1{{\it #1}}

\begin{document}

\title[Linear independence of time frequency translates]{Linear independence of time frequency translates for special configurations}

\author{Ciprian Demeter}
\address{Department of Mathematics, Indiana University, 831 East 3rd St., Bloomington IN 47405}
\email{demeterc@@indiana.edu}

\thanks{The  author is supported by a Sloan Research Fellowship and by NSF Grants DMS-0742740 and 0901208}
\begin{abstract}

We prove that for any 4 points in the plane that belong to 2 parallel lines, there is no linear dependence between the associated time-frequency translates of any nontrivial Schwartz function.
If mild Diophantine properties are satisfied, we also prove  linear independence in the category of $L^2(\R)$ functions.
\end{abstract}
\maketitle

\section{Introduction}
The following conjecture, known as the HRT conjecture appears in \cite{HRT}. See also \cite{H} for an ample discussion on the subject.

\begin{conjecture}
\label{cc1}
Let $(t_j,\xi_j)_{j=1}^n$ be $n\ge 2$ distinct points in the plane. Then there is no nontrivial $L^2$ function  $f:\R\to \C$ satisfying a  nontrivial linear dependence
$$\sum_{j=1}^nd_if(x+t_j)e^{2\pi i\xi_jx}=0,$$
for a.e. $x\in \R$.
\end{conjecture}

The  Conjecture was proved  when  $(t_i,\xi_i)_{i=1}^n$ sit on a lattice \cite{Lin}. See also \cite{BS}, \cite{DG}, for alternative arguments. In particular, this is the case with any 3 points. The conjecture also follows trivially when all points are collinear. No other cases seem to appear in the literature.
The following weaker conjecture has also been circulated (see for example \cite{H}).

\begin{conjecture}
\label{c2}
Let $(t_j,\xi_j)_{j=1}^n$ be $n\ge 2$ distinct points in the plane. Then there is no nontrivial Schwartz function  $f:\R\to \C$ satisfying a  nontrivial linear dependence
$$\sum_{j=1}^nd_if(x+t_j)e^{2\pi i\xi_jx}=0,$$
for a.e. $x\in \R$.
\end{conjecture}

 In light of the discussion above, this conjecture also follows  for the lattice and when the points are collinear. No other result seems to have appeared in the direction of this conjecture.

 We will call an $(n,m)$ configuration, any collection of $n+m$ distinct points in the plane, such that there exist 2 distinct parallel lines such that one of them contains exactly $n$ of the points, and the other one contains exactly $m$ of the points.
Our main results are:
\begin{theorem}
\label{t:main}
Conjecture \ref{c2} holds for all $(1,3)$ and $(2,2)$ configurations.

\end{theorem}

Let $\|x\|$ denote the distance of $x$ to the nearest integer.

\begin{theorem}
\label{TT2}
Conjecture \ref{cc1} holds for  special $(1,3)$ configurations $$(0,0),(1,0), (1,\alpha), (1,\beta)$$

(a) if there exists $\gamma>1$ such that
$$\liminf_{n\to\infty}n^{\gamma}\min\{\|n\frac{\beta}{\alpha}\|,\|n\frac{\alpha}{\beta}\|\}<\infty$$

(b)  if at least one of $\alpha,\beta$ is rational

In either case, no nontrivial solution $f$ can exist satisfying  minimal decay
$$\lim_{|n|\to\infty\atop{n\in\Z}}|f(x+n)|=0,\;\;a.e.\;x$$
\end{theorem}

\begin{theorem}
\label{TT2'}
Conjecture \ref{cc1} holds  for special $(2,2)$ configurations $$(0,0),(1,0), (0,\alpha), (1,\beta)$$

(a) if
$$\liminf_{n\to\infty}n\log n\min\{\|n\frac{\beta}{\alpha}\|,\|n\frac{\alpha}{\beta}\|\}<\infty$$

(b)  if at least one of $\alpha,\beta$ is rational

In either case, no nontrivial solution $f$ can exist satisfying  minimal decay
$$\lim_{|n|\to\infty\atop{n\in\Z}}|f(x+n)|=0,\;\;a.e.\;x$$
\end{theorem}

Unlike the approaches in \cite{HRT}, \cite{Lin}, \cite{BS}, \cite{DG}, the approach here is mostly number theoretical.
An old theorem of Khinchine guarantees that (Lebesgue) almost every $x$ satisfies $$\liminf_{n\to\infty}n\log n\|nx\|<\infty,$$
and thus (via metaplectic transforms) Conjecture \ref{cc1} holds for "almost every"  $(2,2)$ configuration. Note also that Theorem \ref{TT2}(b) answers Conjecture 9.2 (b) from \cite{H}.

Theorem \ref{t:main} is proved by first reducing to special configurations. This is done via applying the area preserving affine transformations -also called {\em metaplectic transforms}- of the plane (such as translations, rotations, shears, and area one rescalings). See Section 2 in \cite{HRT} for a discussion on this.

The key feature of any special $(n,m)$ configuration of points, is the fact that any linear dependence between the corresponding time frequency translates gives rise to a recurrence along $\Z$ orbits $x+\Z$. We use Diophantine approximation to identify appropriate {\em scales}. For each fixed scale, we investigate the recurrence along finite portions of two carefully chosen distinct orbits, with length comparable to the scale.

The (2-2) case is quite simple. The two rotations by $\alpha$ and $\beta$ do not interact strongly, and hence they can be dealt with by different methods. The two orbits are selected in such a way that the trigonometric polynomials associated with $\alpha$ take conjugate values along the two orbits. This means that in absolute value, the contributions coming from these polynomials are  identical, for the two orbits. We will refer to this as the {\em conjugates trick}. The contributions coming from the polynomials associated with $\beta$ are then compared via diophantine approximation and Riemann sums, if $\beta$ is irrational, and using a {\em periodicity argument}, if $\beta$ is rational.

The (1,3) case  is significantly harder, partly because the behavior of the relevant polynomial $p(x,y)=C_0+C_1e(x)+C_2e(y)$ near its zeros is more complicated.
Our argument  relies in part on the almost periodicity for trigonometric polynomials. This in turn has behind the existence of simultaneous approximants $P\in \R$ such that $$P\max\{\|P\alpha\|,\|P\beta\|\}\lesssim 1.$$
The almost periods we get become better as $\frac\alpha\beta$ gets less Diophantine (that is better approximable by rationals).
The main idea of this approach is to compare the products of polynomials along the two orbits via estimates for their arithmetic means (see for example Proposition \ref{P1}). While the products along a fixed orbit can fluctuate a lot, and thus are very difficult to control, proving upper bounds for their averages turns out to be a less complicated proposal.
This requires a deeper understanding of the geometry of the points $(\{n\alpha\},\{n\beta\})$, as observed in the proof of Lemma \ref{l7} and Lemma \ref{l7hrstmdncvjrhufherw;'q,.c}. The key is that the more Diophantine $\frac{\alpha}{\beta}$ is, the more "regular" is the counting measure of the points $(\{n\alpha\},\{n\beta\})$, and this will serve as a compensation.

\section{Proof of Theorem \ref{t:main} for (1,3) configurations}
 Define $[x]$, $\{x\}$, $\|x\|$  to be the integer part, the  fractional part and the distance to the nearest integer of $x$. Let $\langle x\rangle$ denote the unique number in $[-1/2,1/2)$ such that $x-\langle x\rangle$ is an integer. For two quantities $A$, $B$ that vary, we will denote by $A\lesssim B$ or $A=O(B)$ the fact that $A\le C B$ for some universal constant $C$, independent of $A$ and $B$.  In general, $A\lesssim_p B$ means that the implicit constant is allowed to depend on the parameter $p$. If no parameter is specified, the implicit constants are implicitly understood to depend on the (harmless) fundamental parameters introduced in the beginning of the proof of Theorem \ref{t:main}. For a set $A\subset \R$, we will denote by $|A|$ its Lebesgue measure, and if the set is finite, $|A|$ will represent its cardinality. Finally, we define
$e(x):=e^{2\pi i x}$.

We prove a few results that will contribute to the proof of Theorem \ref{t:main}.

\begin{lemma}
\label{L4}
Let $C_0,C_1,C_2\in \C$ be some nonzero $complex$ numbers. The polynomial $p(x,y)=C_0+C_1e(x)+C_2e(y)$ has at most two  real zeros $(\gamma_1^{(j)},\gamma_2^{(j)})\in [0,1)^2$, $j\in\{1,2\}$. There exists $t=t(C_0,C_1,C_2)\in \R\setminus 0$ such that
$$|p(x,y)|\gtrsim_{C_0,C_1,C_2}\min_{j}(\| x-\gamma_1^j+t\langle y-\gamma_2^j\rangle\|+\| x-\gamma_1^j\|^2+\| y-\gamma_2^j\|^2),$$
for each $x,y\in\R$.

\end{lemma}
\begin{proof}
If $|C_0|,|C_1|,|C_2|$ can not form a triangle, then $|p(x,y)|\gtrsim_{C_0,C_1,C_2} 1$ and there is nothing to prove. If $|C_0|,|C_1|,|C_2|$ can  form a triangle, given that the side with length  $|C_0|$ is rigid, there are only two possible ways to construct the other two sides (the two triangles will be symmetric with respect to the side with length $|C_0|$). This justifies the fact that there are at most 2 zeros.

Since $p$, $\|\cdot\|$ and $\langle\cdot \rangle$ are 1 periodic, and since $\langle y\rangle =y$ and $\|y\|=|y|$ near 0, it suffices to prove that
\begin{equation}
\label{equ1}
|p(x,y)|\gtrsim_{C_0,C_1,C_2}|(x-\gamma_1^j)+t(y-\gamma_2^j)|,
\end{equation}
and
\begin{equation}
\label{equ223ds}
|p(x,y)|\gtrsim_{C_0,C_1,C_2}|x-\gamma_1^j|^2+|y-\gamma_2^j|^2,
\end{equation}
for $(x,y)$ in a sufficiently small neighborhood (on $\R$) of $(\gamma_1^j,\gamma_2^j)$.

We distinguish two cases. The {\em non-degenerate case} is when $\frac{C_1e(\gamma_1^j)}{C_2e(\gamma_2^j)}$ is not a real number. This is the same as saying that $|C_0|,|C_1|,|C_2|$  form a non-degenerate triangle. By using a Taylor expansion we get
$$\frac{p(x,y)}{2\pi i}=$$$$C_1e(\gamma_1^j)(x-\gamma_1^j+\pi i(x-\gamma_1^j)^2+O(|x-\gamma_1^j|^3))+C_2e(\gamma_2^j)((y-\gamma_2^j)+\pi i(y-\gamma_2^j)^2+O(|y-\gamma_2^j|^3)).$$Note that
$$|C_1e(\gamma_1^j)(x-\gamma_1^j)+C_2e(\gamma_2^j)(y-\gamma_2^j)|\gtrsim_{C_0,C_1,C_2}((x-\gamma_1^j)^2+(y-\gamma_2^j))^{1/2},$$
and thus $|p(x,y)|\gtrsim_{C_0,C_1,C_2}((x-\gamma_1^j)^2+(y-\gamma_2^j))^{1/2}\ge (x-\gamma_1^j)^2+(y-\gamma_2^j)^2$, for $|x-\gamma_1^j|,|y-\gamma_2^j|$ sufficiently small. Hence \eqref{equ1} also holds with, say, $t=1$.

The {\em degenerate case} is when $|C_0|,|C_1|,|C_2|$  form a degenerate triangle. In this case there is only one zero, call it $(\gamma_1,\gamma_2)$. It follows that $C_0,C_1e(\gamma_1),C_2e(\gamma_2)$ are real multiples of each other. Thus, there must exist two among these three numbers with a positive ratio. There are two cases.

First let us assume $t:=\frac{C_2e(\gamma_2)}{C_1e(\gamma_1)}>0$. Then, proceeding as before,
$$\frac{p(x,y)}{2\pi iC_1e(\gamma_1)}=(x-\gamma_1+t(y-\gamma_2)+O(|x-\gamma_1|^3+|y-\gamma_2|^3))+$$$$\pi i((x-\gamma_1)^2+t(y-\gamma_2)^2+O(|x-\gamma_1^j|^3+|y-\gamma_2|^3))$$
where the first term is the real part, while the second term is the imaginary part. Since $t>0$,
$$(x-\gamma_1)^2+t(y-\gamma_2)^2\gtrsim_t(x-\gamma_1)^2+(y-\gamma_2)^2>>(|x-\gamma_1|^3+|y-\gamma_2|^3),$$
for $|x-\gamma_1|,|y-\gamma_2|$ sufficiently small. If $|(x-\gamma_1)+t(y-\gamma_2)|\ge (x-\gamma_1)^2+(y-\gamma_2)^2,$ then the real part is dominant, and thus
$|p(x,y)|\gtrsim |(x-\gamma_1)+t(y-\gamma_2)|$. If $|(x-\gamma_1)+t(y-\gamma_2)|\le (x-\gamma_1)^2+(y-\gamma_2)^2,$ then the imaginary part is dominant, and thus $|p(x,y)|\gtrsim (x-\gamma_1)^2+(y-\gamma_2)^2$.

The second possibility is that $s:=\frac{C_1e(\gamma_1)}{C_0}>0$. The case $\frac{C_2e(\gamma_2)}{C_0}>0$ is completely symmetric, so we omit it. Note that
$$|p(x,y)|=|C_2+C_0e(-\gamma_2)e(\gamma_2-y)+C_1e(\gamma_1-\gamma_2)e(x-\gamma_1-(y-\gamma_2))|=$$
$$2\pi|C_0e(-\gamma_2)\left((\gamma_2-y)+\pi i(\gamma_2-y)^2+O(|\gamma_2-y|^3)\right)+$$$$C_1e(\gamma_1-\gamma_2)\left((x-\gamma_1-(y-\gamma_2))+\pi i(x-\gamma_1-(y-\gamma_2))^2+O(|x-\gamma_1-(y-\gamma_2)|^3)\right)|$$
Thus,
$$\frac{|p(x,y)|}{|2\pi C_0e(-\gamma_2)|}=$$
$$|\left((\gamma_2-y)+s(x-\gamma_1-(y-\gamma_2))+O_s(|\gamma_2-y|^3+|x-\gamma_1|^3)\right)+$$
$$\pi i\left((\gamma_2-y)^2+s(x-\gamma_1-(y-\gamma_2))^2+O_s(|\gamma_2-y|^3+|x-\gamma_1|^3)\right)|.$$
Using the fact that
$$(\gamma_2-y)^2+s(x-\gamma_1-(y-\gamma_2))^2\gtrsim_s (\gamma_2-y)^2+(x-\gamma_1-(y-\gamma_2))^2\ge \frac{(\gamma_2-y)^2+(x-\gamma_1)^2}{4},$$
\eqref{equ1} and \eqref{equ223ds} follow as before, this time with $t:=-\frac{1+s}{s}$.

\end{proof}

We will for the rest of the paper implicitly assume $\beta>0$.
The following result uses simultaneous diophantine approximation to construct sharp almost periods. The requirement (i) will be needed later, in order to be able to place the generators $x$ and $x+\{P_k\}$ of the two orbits, in the (potentially very small) interval $I$ where the "solution" $f$ is guaranteed to be nonzero.

\begin{lemma}
\label{L1}
Let $\alpha,\beta\in\R$ be two nonzero numbers with $\alpha/\beta$ irrational. Let also $1>s>0$ be fixed. Then there exists a constant $0<D=D(s,\alpha,\beta)<\infty$ and a sequence $N_k$ of positive integers going to infinity such that for each $k\ge 1$

(i) $\{\frac{N_k}{\beta}\}<s$,

(ii) $N_k\|N_k\frac{\alpha}{\beta}\|\le D\min_{1\le n\le N_k}n\|n\frac{\alpha}{\beta}\|$,

(iii) $N_k\|N_k\frac{\alpha}{\beta}\|\le D$

\end{lemma}
\begin{proof}
We have two cases. If (call this the {\em badly approximable regime})
$$\liminf_{N\to\infty}N\|N\frac{\alpha}{\beta}\|>0,$$
then let $\epsilon:=\min_{N\in\N}N\|N\frac{\alpha}{\beta}\|>0.$
Dirichlet's Theorem implies that we can choose a sequence $N_k'\to\infty$  such that
$N_k'\|N_k'\frac{\alpha}{\beta}\|\le 1$ for each $k$. But then
$$N_k'\|N_k'\frac{\alpha}{\beta}\|\le \frac{1}{\epsilon
}\min_{n\in\N}n\|n\frac{\alpha}{\beta}\|.$$
Finally, by pigeonholing, for each $k$ there must exist some $1\le m_k\le 1/s$ such that $N_k:=m_kN_k'$ satisfies (i). It is immediate that
$$N_k\|N_k\frac{\alpha}{\beta}\|\le \frac{1}{s^2\epsilon
}\min_{n\in\N}n\|n\frac{\alpha}{\beta}\|.$$

Assume next that we are in the {\em well approximable regime}, that is
$$\liminf_{N\to\infty}N\|N\frac{\alpha}{\beta}\|=0.$$
This is equivalent with saying the the sequence $a_k$ in the continued fraction of $\alpha/\beta$ is unbounded.
Let $(p_k,N_k')$ be the sequence of best approximants for $\frac{\alpha}{\beta}$, ordered such that $N_k'$ is increasing.
Thus
$$|\frac{\alpha}{\beta}-\frac{p_k}{N_k'}|\le\frac{1}{N_k'N_{k+1}'}.$$ Recall that $N_{k+1}'\ge a_kN_k'$, and thus $sN_{k+1}'>N_k'$ for each $k\in E$, where $E$ is infinite.  It is known (see for example Theorem 7.13 in \cite{Nie}) that for each $k$
$$N_k'\|N_k'\frac{\alpha}{\beta}\|= \min_{1\le n< N_{k+1}'}n\|n\frac{\alpha}{\beta}\|.$$Let $\pi:\N\to E$ be an increasing bijection.
Choose as before  $1\le m_k\le 1/s$ such that $N_k:=m_kN_{\pi(k)}'$ satisfies (i). Note that $N_k<N_{\pi(k)+1}'$, and thus (ii)-(iii) follow as before.

\end{proof}

The next lemma will be needed in the proof of Proposition \ref{P1}.

\begin{lemma}\label{l7}
Let $C_0,C_1,C_2\in \C$ be some nonzero $complex$ numbers. Let $\alpha,\beta$ be some nonzero real numbers. Define
$$P(x)=C_0+C_1e(\alpha x)+C_2e(\beta x).$$ Let $(N_k)$ be a sequence such that (ii) and (iii) in Lemma \ref{L1} hold. Define $\frac1{M_k}:=\frac{N_k\|N_k\frac{\alpha}{\beta}\|}{D},$
and let $P_k:=\frac{N_k}{\beta}$.
Then for each $k$ and each $\delta>0$, there exists an exceptional set $E_{k,\delta}\subset [0,1]$ such that
$$|E_{k,\delta}|<\delta$$
and
$$\frac{1}{M_kP_k}\sum_{n=0}^{[P_k]-1}\frac{1}{|P(x+n)|}\lesssim_{\delta,C_0,C_1,C_2,\alpha,\beta} \log P_k,$$
for each $x\in [0,1]\setminus E_{k,\delta}$. \end{lemma}
\begin{proof}
Let $(\gamma_1,\gamma_2)\in[0,1]^2$ be a real zero of the polynomial $p(x,y)=C_0+C_1e(x)+C_2e(y)$. Define
$$A_n(x):=\|\alpha (x+n)-\gamma_1\|^2+\|\beta (x+n)-\gamma_2\|^2$$
By Lemma \ref{L4}, it suffices to find an exceptional set with $|E_{k,\delta}|\le \frac{\delta}{2},$ such that
\begin{equation}
\label{equ:er5ahyso}
\frac{1}{P_kM_k}\sum_{n=0}^{[P_k]-1}\frac{1}{A_n(x)}\lesssim_{\delta,\alpha,\beta} \log P_k,
\end{equation}
for each $x\in [0,1]\setminus E_{k,\delta}$.
The heuristics behind the proof is that if $\alpha/\beta$ is less Diophantine, then the estimate above holds because $M_k$ is large, while if $\alpha/\beta$ is Diophantine, we win because of extra regularity of the counting measure for the points $(n\alpha,n\beta)$.

The key observation  is that each strip
$$S_a:=\{(x,y):a-\frac{1}{4M_kP_k}\le \beta x-\alpha y< a+\frac{1}{4M_kP_k}\}$$
 contains at most $O_{\beta}(1)$ points $(\{-n\alpha+\gamma_1\},\{-n\beta+\gamma_2\})$, $0\le n\le [P_k]-1$. Indeed, assume for contradiction that some $S_a$ contains  both  $(\{-n\alpha+\gamma_1\},\{-n\beta+\gamma_2\})$ and $(\{-m\alpha+\gamma_1\},\{-m\beta+\gamma_2\})$, with $[P_k]-\frac1\beta>|m-n|>\frac{2}{\beta}$. It follows that $$|\beta([-n\alpha+\gamma_1]-[-m\alpha+\gamma_1])-\alpha([-n\beta+\gamma_2]-[-m\beta+\gamma_2])| < \frac{1}{M_kP_k}.$$
Note that $L:=|[-n\beta+\gamma_2]-[-m\beta+\gamma_2]|\not=0$ and $L\le |-n\beta+m\beta|+1\le N_k$. The above implies that $\|L\frac{\alpha}{\beta}\|<\frac{1}{M_kP_k\beta}\le\frac{1}{M_kL}$, which contradicts  (ii)  in Lemma \ref{L1}.

Let $\mathcal C$ be the collection of all points $\{-n\alpha+\gamma_1\}$ with $0\le n\le [P_k]-1$. Let $E_\delta'$ be the set of $x\in [0,1)$ such that the distance from $\{\alpha x\}$ to $\mathcal C$ is less than $\frac{\delta|\alpha|}{20P_k}$. Then $|E_\delta'|<\delta/10$.

Let $E_\delta''$ consist of those $x\in [0,1)$ with  $\min(\|\alpha x\|,\|\beta x\|)\le \frac{\delta(|\alpha|+|\beta|)}{100}$. Then $|E_\delta''|\le\frac{\delta}{10}$.

Note that if $u,v\in [0,1)$ with $\|u\|>\epsilon$ for some $\epsilon>0$, then
\begin{equation}
\label{e:199056484637}
\|u-v\|\gtrsim_\epsilon |u-v|.
\end{equation}

Split $[0,1)\setminus E_\delta''$ in intervals $H\in\mathcal H$, such that for each $H$, the integer parts $[\alpha x]$ and $[\beta x]$ are both constant, when $x$ varies through $H$. Thus, the points $(\{\alpha x\},\{\beta x\})$, $x\in H$ sit on a fixed line $\beta x-\alpha y=c_H$. Note that $\mathcal H$ contains $O_{\alpha,\beta}(1)$ intervals.

Let $H\in\mathcal H$. Then, using \eqref{e:199056484637},
$$\int_{H\setminus E_{\delta'}}\frac{1}{P_kM_k}\sum_{n=0}^{[P_k]-1}\frac{1}{A_n(x)}d\lesssim_{\alpha,\beta,\delta}$$$$\int_{H\setminus E_{\delta'}}\frac{1}{P_kM_k}\sum_{n=0}^{[P_k]-1}\frac{1}{|\{\alpha x\}-\{-n\alpha+\gamma_1\}|^2+|\{\beta x\}-\{-n\beta+\gamma_2\}|^2}dx.$$
The key observation implies that there are $O_\beta(1)$ points $(\{-n\alpha+\gamma_1\},\{-n\beta+\gamma_2\})$ in each strip $St_j:=\{(x,y):c_H+\frac{j}{2M_kP_k}\le \beta x-\alpha y< c_H+\frac{j+1}{2M_kP_k}\}$, with, say,  $-10M_kP_k\le j\le 10M_kP_k$. Call $B_j$ the set of $n$ corresponding to these points.
We first evaluate for $j\notin\{-1,0\}$,
$$\int_{H}\sum_{n=0\atop{n\in B_j}}^{[P_k]-1}\frac{1}{|\{\alpha x\}-\{-n\alpha+\gamma_1\}|^2+|\{\beta x\}-\{-n\beta+\gamma_2\}|^2}dx.$$
Since $(\{\alpha x\},\{\beta x\})_{x\in H}$ belong to a line segment, and since the points in $St_j$ belong to a strip parallel to and at distance at least $|j|/2M_kP_k$ from this line segment,
a simple change of coordinates shows that the integral above is dominated by $O_{\alpha,\beta}(\int_0^1\frac{1}{x^2+\left(\frac{|j|}{M_kP_k}\right)^2}dx)=O_{\alpha,\beta}(\frac{M_kP_k}{|j|})$. Of course, only at most $P_k$ values of $j$ can contribute, and thus
$$\int_{H}\frac{1}{M_kP_k}\sum_{n=0\atop{n\in \bigcup_{j\notin\{0,1\}}B_j}}^{[P_k]-1}\frac{1}{|\{\alpha x\}-\{-n\alpha+\gamma_1\}|^2+|\{\beta x\}-\{-n\beta+\gamma_2\}|^2}d\lesssim_{\alpha,\beta}$$
$$\frac{1}{M_kP_k}\sum_{j\notin\{0,1\}\atop{St_j\not=\emptyset}}\frac{M_kP_k}{|j|}\lesssim \log P_k.$$

On the other hand, if $n\in B_0\cup B_{-1}$, then we use the fact that if $x\notin E_{\delta}'$ then
$$|\{\alpha x\}-\{-n\alpha+\gamma_1\}|^2+|\{\beta x\}-\{-n\beta+\gamma_2\}|^2\gtrsim \frac{\delta^2|\alpha|^2}{P_k^2}.$$Using this, the fact that $B_0\cup B_{-1}$ contains $O_\beta(1)$ points and the fact that $M_k\ge 1$, we get via a change of variables that
$$\int_{H \setminus E_{\delta}'}\frac{1}{M_kP_k}\sum_{n=0\atop{n\in B_0\cup B_{-1}}}^{[P_k]-1}\frac{1}{|\{\alpha x\}-\{-n\alpha+\gamma_1\}|^2+|\{\beta x\}-\{-n\beta+\gamma_2\}|^2}dx$$
is dominated by $\frac{1}{M_kP_k}O(\int_{\frac{\delta^2|\alpha|^2}{P_k^2}}^1\frac{1}{x^2}dx))=O_{\alpha,\beta}(\frac{P_k}{\delta M_kP_k})=O_{\alpha,\beta,\delta}(1).$
Thus
$$\int_{[0,1)\setminus (E_{\delta'}\cup E_\delta'')}\frac{1}{M_kP_k}\sum_{n=0}^{[P_k]-1}\frac{1}{|\{\alpha x\}-\{-n\alpha+\gamma_1\}|^2+|\{\beta x\}-\{-n\beta+\gamma_2\}|^2}dx=$$$$\sum_{H\in\mathcal H}\int_{H\setminus E_{\delta'}}\frac{1}{M_kP_k}\sum_{n=0}^{[P_k]-1}\frac{1}{|\{\alpha x\}-\{-n\alpha+\gamma_1\}|^2+|\{\beta x\}-\{-n\beta+\gamma_2\}|^2}dx$$$$\lesssim_{\alpha,\beta,\delta} \log P_k,$$
since there are $O_{\alpha,\beta}(1)$ intervals $H$. It follows  that
$$|E_\delta''':=\{x\in [0,1)\setminus (E_{\delta'}\cup E_\delta''):\frac{1}{M_kP_k}\sum_{n=0}^{[P_k]-1}\frac{1}{|P(x+n)|}\ge C_{\alpha,\beta,\delta}\log P_k\}|\le \delta/10,$$
if $C_{\alpha,\beta,\delta}$ is chosen large enough.
Finally, define $E_{k,\delta}:=E_\delta'\cup E_\delta''\cup E_\delta'''$.
\end{proof}

The next proposition captures the almost orthogonality phenomenon. The idea is to avoid estimating products along individual orbits (which seems nearly impossible), but rather to compare products along two orbits.

\begin{proposition}
\label{P1}
Let $(N_k)$ be a sequence such that (ii) and (iii) in Lemma \ref{L1} hold, and define $P_k:=\frac{N_k}{\beta}$. Let $C_0,C_1,C_2\in\C$ and $\alpha,\beta\in \R$ be nonzero and define $P(x)=C_0+C_1e(\alpha x)+C_2e(\beta x)$. Given $\delta>0$,  there exists $E_{k,\delta}\subset [0,1]$ such that $|E_{k,\delta}|\le \delta$, and there exists a  universal constant $L=L(\delta,C_0,C_1,C_2,\alpha,\beta)$ such that  for each $x,y$ satisfying $x\in [0,1)\setminus E_{k,\delta}$ and $x=y+P_k$, we have
$$|\prod_{n=0}^{[P_k]-1}P(y+n)|\le P_k^L |\prod_{n=0}^{[P_k]-1}P(x+n)|.$$
\end{proposition}
\begin{proof}
Let as before
$$\frac1{M_k}:=\frac{P_k\|P_k\frac\alpha\beta\|}{D}<1.$$
Then, we have
$$|e(\alpha x)-e(\alpha y)|=|e(P_k\alpha)-1|\le \frac{10D}{P_kM_k}\;\;\text{ and }\;\;|e(\beta x)-e(\beta y)|=|e(P_k\beta)-1|\le \frac{10D}{P_kM_k}.$$
Thus, for each $n\in \N$
$$|P(y+n)|\le |P(x+n)|+\frac{CD}{M_kP_k},$$
for some universal constant $C:=100(|C_1|+|C_2|)$.

Use the fact $a+b\le ae^{\frac{b}{a}}$ for each $a,b>0$ to get
$$|P(y+n)|\le |P(x+n)|e^{\frac{CD}{M_kP_k|P(x+n)|}},$$
and thus
$$|\prod_{n=0}^{[P_k]-1}P(y+n)|\le |\prod_{n=0}^{[P_k]-1}P(x+n)|e^{\frac{CD}{M_kP_k}\sum_{n=0}^{[P_k]-1}\frac1{|P(x+n)|}}.$$
The result now follows from Lemma \ref{l7}.
\end{proof}

\begin{proof}
(of Theorem \ref{t:main} for (1,3) configurations)

It was proved in \cite{HRT} that the linear independence of the time-frequency translates is preserved under area preserving affine transformations of the plane, both for $L^2$ and for Schwartz functions. Thus, it is easy to see that any (1,3) configuration  can be reduced to a special (1,3) configuration like $(0,0),(1,0),(1,\alpha),(1,\beta)$, for some nonzero $\alpha,\beta$, with $\alpha\not=\beta$. Indeed, apply first a rotation, then a joint rescaling of the time and frequency axes so that the distance between the line containing the 3 points and the remaining point becomes 1, followed by a translate along the frequency axis and then by a vertical shear. It is also clear that it is precisely the $(1,3)$ configurations (and they only), that can be reduced to the special configuration from above. This follows since any linear transformation of the plane maps collinear points into collinear points.

We can also assume that the 4 points do not sit on a lattice, since that case is covered by the work in \cite{Lin}. In particular, we assume $\alpha/\beta$ is irrational.

Assume now for contradiction that there is some nontrivial Schwartz $f:\R\to\C$ such that for a.e. $x$, for some fixed nonzero constants $C_i\in\C$
$$f(x+1)=f(x)(C_0+C_1e(\alpha x)+C_2e(\beta x)).$$
 We denote as before
$P(x)=C_0+C_1e(\alpha x)+C_2e(\beta x)$.

Let $(I,S,d)$ be a triple such that $d>0$, $I\subset [0,1)$ is an interval, $S\subset I$ is a set with full Lebesgue measure $|I|$,  such that
\begin{equation}
\label{e4}
\lim_{|n|\to\infty\atop_{n\in\Z}}|n|^Cf(x+n)=0, \text{ for each } x\in S \text{ and each }C>0,
\end{equation}
\begin{equation}
\label{e5}
d<|f(x)|<d^{-1},\;\text{ for each }x\in S\cup S+1,
\end{equation}
\begin{equation}
\label{e6}
f(x+1)=f(x)(C_0+C_1e(\alpha x)+C_2e(\beta x)),\;\text{ for each } x\in S+\Z.
\end{equation}

We will refer to $C_0,C_1, C_2,I,S,d,\alpha,\beta$ as {\em fundamental parameters}. Implicit constants in inequalities involving $\lesssim$ are allowed to depend on these parameters, since they can be thought of as being fixed for the rest of the argument.

Apply Lemma \ref{L1} with $s:=|I|/10$, to get  a constant $D=D(|I|,\alpha,\beta)>0$ and a sequence $N_k$ of positive integers going to infinity such that for each $k\ge 1$

(i) $\{P_k:=\frac{N_k}{\beta}\}<|I|/10$,

(ii) $N_k\|N_k\frac{\alpha}{\beta}\|\le D\min_{1\le n\le N_k}n\|n\frac{\alpha}{\beta}\|$,

(iii) $N_k\|N_k\frac{\alpha}{\beta}\|\le D$

By Proposition \ref{P1} with $\delta=\frac{|I|}{100}$, for each  $x\in [0,1)\setminus E_{k,\delta}$ and $x=y+P_k$ we have
\begin{equation}
\label{e8}
|\prod_{n=0}^{[P_k]-1}P(y+n)|\le P_k^L |\prod_{n=0}^{[P_k]-1}P(x+n)|.
\end{equation}
Note that $(S\setminus E_{k,\delta})\cap (\{P_k\}+ S)\not=\emptyset$, since the intersection of each of the 2 sets with $I$ is large. Pick any $x_k\in S\setminus E_{k,\delta}$, such that  also $x_k-\{P_k\}\in S$.
Define $x_k':=x_k-\{P_k\}$. We will now  argue that \eqref{e4} can not hold for both $x_k$ (with $n\to\infty$) and $x_k'$ (with $n\to-\infty$).
Let $y_k=x_k'-[P_k]=x_k-P_k$. Note that
$$\prod_{n=0}^{[P_k]-1}P(y_k+n)=\prod_{n=1}^{[P_k]}P(x_k'-n),$$
and thus, by \eqref{e8} we get
\begin{equation}
\label{e12}
|\prod_{n=1}^{[P_k]}P(x_k'-n)|\le P_k^L |\prod_{n=0}^{[P_k]-1}P(x_k+n)|.
\end{equation}
Recall that we have
\begin{equation}
\label{e13}
f(x_k+[P_k])=f(x_k){\prod_{n=0}^{[P_k]-1}P(x_k+n)}
\end{equation}
\begin{equation}
\label{e14}
f(x_k'-[P_k])=f(x_k'){\prod_{n=1}^{[P_k]}P(x_k'-n)}.
\end{equation}
Using \eqref{e5}  and \eqref{e12}, we get that
\begin{equation}
\label{e16}
|f(x_k+[P_k])f(x_k'-[P_k])|\gtrsim P_k^{-L},
\end{equation}
with an implicit constant depending only on the fundamental parameters (in particular, independent of $k$).
This clearly contradicts \eqref{e4}, if we let $k\to\infty$.
\end{proof}
\section{Proof of Theorem \ref{t:main} for (2,2) configurations}

The reader can easily check that the argument for $(1,3)$ configurations presented above also works here. In particular, the existence of large $P\in\R$ satisfying $$\max\{\|P\alpha\|,\|P\beta\|\}\lesssim 1/P$$ implies that one can run the almost periodicity argument.

We also present an alternative, simpler argument, which does not work in the (1,3) case. As before, by using metaplectic transformations we can reduce to the case of special (2,2) configurations.
Assume there is a continuous function  $f:\R\to \C$ such that $f$ is nonzero on some interval $I\in [0,1)$, satisfying a weaker assumption (minimal decay)
\begin{equation}
\label{eeeeeeeeeeeeeeeeeeeeee:ejhdfui34yr348r85ug}
\lim_{|n|\to\infty\atop_{n\in\Z}}f(x+n)=0
\end{equation}
for all $x\in [0,1]$,
and\begin{equation}
\label{eeeeeeeeeeeeeeeeeeeeee:ejhdfui34yr348r85ug'}
f(x+1)(A+Be(\alpha x))=f(x)(C+De(\beta x)),\end{equation}
for all $x\in\R$, for some fixed $A,B,C,D\in\C$, $\alpha,\beta\in\R$, none of them zero. We can also assume that $\alpha$ and $\beta$ are rationally independent ($\frac{\beta}{\alpha}\not\in\Q$), by otherwise invoking the lattice case. Note that we assume far less, namely minimal (rather than Schwartz) decay. Thus, the result we prove for special (2,2) configurations is in some sense best possible.
\bigskip

First, let us deal with the case when both $\alpha,\beta$ are irrational.
We can trivially assume $A$ is real. Write
$$P(x)=A+Be(\alpha x)\;\;\text{ and }\;\;Q(x)=C+De(\beta x).$$

We first prove that  $|A|\not=|B|$. Indeed, assume for contradiction that  $|A|=|B|$. Then $P$ has real zeros of the form
$$x_n=\omega+\frac{n}{\alpha},\;n\in\Z,$$
for some $\omega\in\R$. We split the argument in a few steps.
\medskip

Step 1. Here we observe that $Q$ must also have real zeros. If this is not the case, let $n$ be such that
$$x_n>0$$
and (since $\frac1\alpha$ is irrational)$$\{x_n\}\in I,\text{ so }f(\{x_n\})\not=0.$$
Since
$$f(x_n+1)P(x_n)=f(x_n)Q(x_n)$$
we find that $f(x_n)=0$. Using the recurrence, and inductive argument shows that $f(x_n-m)=0$ for each $m\in\N$, leading to the contradiction $f(\{x_n\})=0$. This ends Step 1.
\medskip

Step 2. Since we now know that $Q$ has real zeros, they must be of the form
$$y_m=\omega'+\frac{m}{\beta},\;m\in\Z,$$
for some $\omega'\in\R$. If the zeros of $P$ and $Q$ never share a $\Z-$orbit then we reason like in Step 1 to reach a contradiction. Even stronger, if at most one $x_n$ shares a $\Z-$orbit with some $y_m$, the argument in Step 1 still leads to a contradiction.
\medskip

Step 3. We can thus assume for the rest of the argument that there are $n\not=n'$ and $m,m'$ such that
$$x_n-y_m,\;x_{n'}-y_{m'}\in\Z.$$
Taking the difference it follows that there are $N,M,k\in\Z$ so that
\begin{equation}
\label{ fhvurthgiut8gprioivftogup}
\frac{N}{\alpha}+\frac{M}{\beta}=k,
\end{equation}
Note that our original assumption that $\alpha,\beta$ are rationally independent does not prevent this scenario from happening, so we must rule it out by a more careful analysis.

As both $\alpha$ and $\beta$ are irrational, all $N,M$ satisfying  \eqref{ fhvurthgiut8gprioivftogup} must be of the form $lN_0,lM_0$ for some fixed $N_0,M_0\in\Z\setminus\{0\}$ and arbitrary $l\in\Z$. It further follows that all $n,m$ satisfying $x_n-y_m\in\Z$ are of the form
$$\begin{cases}n=n_0+lN_0\\m=m_0+lM_0,\end{cases}$$
for some fixed $n_0,m_0,N_0,M_0\in\Z$ with $N_0,M_0\not=0$ and arbitrary $l\in\Z$. Note that for each $x_n$ there is at most one $y_m$ such that $x_n,y_m$ share an orbit. In the next step, we will argue based on the ordering of $0,x_n,y_m$ to reach the final contradiction.
\medskip

Step 4. Set $A=\omega+\frac{n_0}{\alpha}$, $B=\omega'+\frac{m_0}{\beta}$. Then the zero $X_l:=A+l\frac{N_0}{\alpha}$ of $P$ only shares the orbit with one zero of $Q$, namely  $Y_l:=B+l\frac{M_0}{\beta}$. Since $\frac{N_0}{\alpha}\not=\frac{M_0}{\beta}$, by choosing $l$ either positive or negative, we can ensure that one of the following three scenarios holds (the analysis is driven by the comparison between the signs of $\frac{N_0}\alpha,\frac{M_0}\beta$ and their magnitudes)

$\bullet\;$ $0<X_l<Y_l$ (for all large enough or small enough $l\in\Z$)

$\bullet\;$ $Y_l<0<X_l$ (for all large enough or small enough $l\in\Z$)

$\bullet\;$ $X_l<Y_l<0$ (for all large enough or small enough $l\in\Z$).

In each case, ar argument like in Step 1 applied to the orbit of $X_l$ and $Y_l$ will force $f$ to be zero at $v_l=\{X_l\}=\{Y_l\}$. Let us see this in the first scenario. Fix such an $l$. By ergodicity, we can also assume that
$$v_l\in I.$$

Using the recurrence and the fact that $Q$ is nonzero on the orbit before $Y_l$, we first find that $f(v_l+1)\not=0$, then $f(v_l+2)\not=0$, all the way up to $f(X_l)\not=0$. Using the recurrence one more time leads to the contradiction
$$0=f(X_l+1)P(X_l)=f(X_l)Q(X_l)\not=0.$$

\medskip

Thus, we have proved that $\inf_{x\in\R}|P(x)|>0$. Similarly, it follows that   $|C|\not=|D|$, and thus $\inf_{x\in\R}|Q(x)|>0$.
The functions $\psi(x)=\ln|A+Be(x)|$ and $\phi(x)=\ln|C+De(x)|$ are now guaranteed to be continuous on $[0,1)$.

Due to  \eqref{eeeeeeeeeeeeeeeeeeeeee:ejhdfui34yr348r85ug}, it follows that for each $x,z\in I$
\begin{equation}
\label{1}
\lim_{N\to\infty}\left(\sum_{n=1}^N\phi(\beta x+\beta n)-\sum_{n=1}^N\psi(\alpha x+\alpha n)\right)= -\infty
\end{equation}
and
\begin{equation}
\label{2}
\lim_{N\to\infty}\left(\sum_{n=-N}^{-1}\phi(\beta z+\beta n)-\sum_{n=-N}^{-1}\psi(\alpha z+\alpha n)\right)\to +\infty
\end{equation}

 Let $p_k,q_k$ relatively prime, $q_k\to\infty$ such that
$$\left|\beta-\frac{p_k}{q_k}\right|\le \frac{1}{q_k^2}.$$
Note that  $n\beta$ is very regular, that is
$$|n\beta-\frac{np_k}{q_k}|\le \frac{1}{q_k},\;\;-q_k\le n\le q_k,$$
where $(np_k \mod q_k)_{n=1}^{q_k}$ cover all the residues mod $q_k$. Using this and  $|\phi'|\gtrsim 1$, we get
$$
|\sum_{n=1}^{q_k}\phi(\beta y+\beta n)-q_k\int_{0}^1 \phi|=O(1)
$$
and
$$
|\sum_{n=-q_k}^{-1}\phi(\beta y+\beta n)-q_k\int_{0}^1 \phi|=O(1)
$$
for each $y\in [0,1]$ (use Riemann sums). An immediate corollary is that for each $x,z\in I$
\begin{equation}
\label{4}
|\sum_{n=1}^{q_k}\phi(\beta x+\beta n)-\sum_{n=-q_k}^{-1}\phi(\beta z+\beta n)|=O(1)
\end{equation}

Let $B=re(\theta)$, with $r>0$. By invoking Birkhoff's pointwise ergodic theorem for the function $1_I$, there exists $x\in I$ and some $n'\in \N$ such that $z:=\{-x-\frac{2\theta}{\alpha}+n'\alpha^{-1}\}\in I$. Let $y:=-x-\frac{2\theta}{\alpha}+n'\alpha^{-1}$ and $m=y-z$.

 The point of this selection is that for each $n$,
$A+Be(\alpha y-n\alpha )$ and $A+Be(\alpha x+n\alpha)$ are complex conjugates
and thus
$$\sum_{n=-N+m}^{-1+m}\psi(\alpha z+\alpha n)=\sum_{n=1}^{N}\psi(\alpha x+\alpha n).$$

This implies that for each $N$,
\begin{equation}
\label{5}
|\sum_{n=-N}^{-1}\psi(\alpha z+\alpha n)-\sum_{n=1}^{N}\psi(\alpha x+\alpha n)|=O(1).
\end{equation}
But now it immediately follows that \eqref{1}-\eqref{5} can not simultaneously hold.
\medskip

If $\alpha$ is irrational and $\beta\in\Q$, apply the conjugates trick to $P$ as above, and use the periodicity of $Q$, like in the proof of Theorem \ref{TT2'} (b) described in Section \ref{S4}.

\section{Proof of Theorem \ref{TT2}}
Let us first prove part (a) of the theorem. Part (b) is proved in the end of this section.
 The hardest case is if $\frac\alpha\beta$ is {\em Diophantine}. This means that there exists $\gamma>1$ and there exists  $\epsilon>0$ small enough to satisfy, say, $\frac{\epsilon}{\gamma+\epsilon-1}\le \frac1{10}$, such that
\begin{equation}
\label{xcniery8u3ifhtgp[r0-9}
\liminf_{n\to\infty}n^{\gamma}\|n\frac{\alpha}{\beta}\|<\infty
\end{equation}
and in addition
\begin{equation}
\label{xcniery8u3ifhtgp[r0-9ufts}
\inf_{n\in\N}n^{\gamma+\epsilon}\|n\frac{\alpha}{\beta}\|:=D_\epsilon>0.
\end{equation}
The easier case is when $\frac\alpha\beta$ is {\em Liouville}, that is if for each $\eta>1$
$$\liminf_{n\to\infty}n^{\eta}\|n\frac{\alpha}{\beta}\|<\infty.$$
\begin{lemma}\label{l7hrst}
Let $x_1,x_2,\ldots,x_N$ be $N$ not necessarily distinct real numbers.
Then for each $N\in\N$ and each $\delta>0$, there exists an exceptional set $E_{N,\delta}\subset [0,1]$ such that
$$|E_{N,\delta}|\le \delta,$$
\begin{equation}
\label{xcniery8u3ifhtgp[r0-9dvruyfg73t895ugijk}
\frac{1}{N}\sum_{n=1}^{N}\frac{1}{\|x-x_n\|}\lesssim_{\delta} \log N,
\end{equation}
\begin{equation}
\label{xcniery8u3ifhtgp[r0-9dvruyfg73t895ugijk'}
\frac{1}{N^2}\sum_{n=1}^{N}\frac{1}{\|x-x_n\|^2}\lesssim_{\delta} 1,
\end{equation}
for each $x\in [0,1]\setminus E_{N,\delta}$.
\end{lemma}
\begin{proof}
Since $y\mapsto \|y\|$ is 1 periodic, we can assume all $x_i$ are in $[0,1]$. Since $\|x-x_i\|\gtrsim \min\{|x-x_i|,|x-(1-x_i)|\}$, by doubling the number of points if necessary, we can replace $\|x-x_i\|$ with $|x-x_i|$. Define
$$U_{N,\delta}:=\bigcup_{i=1}^{N}[x_i-\frac{\delta}{10N},x_i+\frac{\delta}{10N}],$$
and note that
$$\int_{[0,1]\setminus U_{N,\delta}}\sum_{i=1}^{N}\frac{1}{|x-x_i|}d\lesssim N(\log N+\log \frac1\delta),$$
$$\int_{[0,1]\setminus U_{N,\delta}}\sum_{i=1}^{N}\frac{1}{|x-x_i|^2}d\lesssim N^2\delta^{-1}.$$
The result now follows from Cebysev's inequality.
\end{proof}

\begin{lemma}\label{l7hrstmdncvjrhufherw;'q,.c}
Let $\alpha,\beta$ be some nonzero real numbers satisfying \eqref{xcniery8u3ifhtgp[r0-9ufts} above. Define also $M:=N^{\frac{(\gamma-1)(\gamma+\epsilon)}{\gamma+\epsilon-1}}D_\epsilon^{\frac{1}{\gamma+\epsilon-1}}$.
Then for each $\xi\in\R$ and $N\in\N$
$$\frac{1}{N^\gamma}\sum_{n=0\atop{\|n\frac{\alpha}{\beta}-\xi\|\ge \frac{1}{M}}}^{N-1}\frac{1}{\|n\frac{\alpha}{\beta}-\xi\|}\lesssim_{\alpha,\beta,\gamma,\epsilon}1.$$

\end{lemma}
\begin{proof}
The crucial observation is that for each $\frac{D_\epsilon}{2^{\gamma+\epsilon}}\ge R\ge \frac{1}{M}$, the $\|\cdot\|$-ball
$$B(\xi,R):=\{x:\|x-\xi\|<R\}$$
contains at most $N(\frac RD_\epsilon)^{\frac1{\gamma+\epsilon}}$ points $n\frac{\alpha}{\beta}$, with $0\le n\le N-1$. Indeed, assume for contradiction that $N(\frac RD_\epsilon)^{\frac1{\gamma+\epsilon}}+1$ such points are contained in the ball. Then,  two of these points would correspond to some $0\le n<m\le N-1$ with $|n-m|\le (\frac{D_\epsilon}{R})^{\frac1{\gamma+\epsilon}}$. Thus,
$$(m-n)^{\gamma+\epsilon}\|(m-n)\frac{\alpha}{\beta}\|\le D_\epsilon,$$
which contradicts \eqref{xcniery8u3ifhtgp[r0-9ufts}.

In particular, for each $\frac{D_\epsilon}{2^{\gamma+\epsilon}}\ge 2^{-j}\ge \frac1M$ we have  $O(N(\frac1{2^jD_\epsilon})^{\frac1{\gamma+\epsilon}})$ points in the ball $B_j:=B(\xi,2^{-j})$.
Define for such a $j$
$$S_j:=\{0\le n\le N-1:n\frac\alpha\beta\in B_{j+1}\setminus B_j\}.$$
Then, using the fact that there are at most $N$ points outside the ball $B(\xi,\frac{D_\epsilon}{2^{\gamma+\epsilon}})$, we get
\begin{equation*}
\label{e1}
\frac{1}{N^\gamma}\sum_{n=0\atop\|n\frac{\alpha}{\beta}-\xi\|\ge \frac{1}{M}}^{N-1}\frac{1}{\|n\frac{\alpha}{\beta}-\xi\|}\le
\frac{1}{N^\gamma}\left(\sum_{\frac{D_\epsilon}{2^{\gamma+\epsilon}}\ge 2^{-j}\ge \frac{1}{M}}\sum_{n\in S_j}\frac{1}{\|n\frac{\alpha}{\beta}-\xi\|}+N\frac{2^{\gamma+\epsilon}}{D_\epsilon}\right)
\end{equation*}
$$\lesssim_{D_\epsilon,\gamma} \frac{1}{N^\gamma}\sum_{ 2^{-j}\ge \frac{1}{M}}2^jN(\frac1{2^jD_\epsilon})^{\frac1{\gamma+\epsilon}}\lesssim_{\alpha,\beta,\gamma,\epsilon} 1.$$
\end{proof}

We have the following analogue of Lemma \ref{l7}
\begin{lemma}\label{l7'}
Let $C_0,C_1,C_2\in \C$ and $\alpha,\beta$ be some nonzero numbers. Let $\gamma$ satisfy \eqref{xcniery8u3ifhtgp[r0-9ufts} if $\alpha/\beta$ is Diophantine and let $\gamma=2$ if $\alpha/\beta$ is Liouville. Define
$$P(x)=C_0+C_1e(\alpha x)+C_2e(\beta x).$$
Then for each $N\in \N$ and  $\delta>0$, there exists an exceptional set $E_{N,\delta}\subset [0,1]$ such that
$$|E_{N,\delta}|<\delta$$
and
$$\frac{1}{N^\gamma}\sum_{n=0}^{N-1}\frac{1}{|P(x+n)|}\lesssim_{\delta,C_0,C_1,C_2,\alpha,\beta,\epsilon} 1,$$
for each $x\in [0,1]\setminus E_{N,\delta}$.
\end{lemma}
\begin{proof}

Let $(\gamma_1,\gamma_2)$ be a zero of the polynomial $p(x,y)=C_0+C_1e(x)+C_2e(y)$, and let $t$ be the real number guaranteed by Lemma \ref{L4}. Define
$$A_n(x):=\|\alpha (x+n)-\gamma_1+t\langle\beta (x+n)-\gamma_2\rangle\|+\|\alpha (x+n)-\gamma_1\|^2+\|\beta (x+n)-\gamma_2\|^2$$
By Lemma \ref{L4}, it suffices to find an exceptional set with $|E_{N,\delta}|\le \frac\delta2,$ such that
\begin{equation}
\label{equ:er5ahyso'}
\frac{1}{N^\gamma}\sum_{n=0}^{N-1}\frac{1}{A_n(x)}\lesssim_{\delta,t,\alpha,\beta} 1,
\end{equation}
for each $x\in [0,1]\setminus E_{N,\delta}$.

We first analyze the case when $\alpha/\beta$ is Diophantine.
We distinguish two subcases. First, let us analyze the case $\alpha+t\beta\not=0$. In this case,
$$\|\alpha (x+n)-\gamma_1+t\langle \beta (x+n)-\gamma_2\rangle\|=\|(\alpha+t\beta)x+(\alpha+t\beta)n-\gamma_1-t\gamma_2-t[\beta (x+n)-\gamma_2]+mt\|,$$
where $m=-1$ if  $\{\beta (x+n)-\gamma_2\}>1/2$ and $m=0$ otherwise. Note that the set of points
$$S:=$$$$\{(\alpha+t\beta)n-\gamma_1-t\gamma_2-t[\beta (x+n)-\gamma_2]+mt:\;x\in[0,1],\; 0\le n\le N-1,\;m\in\{0,-1\}\}$$
has $O_\beta(N)$ elements. \eqref{xcniery8u3ifhtgp[r0-9dvruyfg73t895ugijk} implies the existence of $E_{N,\delta}$ with $|E_{N,\delta}|<\delta/2$ and
$$\frac{1}{N}\sum_{y\in S}\frac{1}{\|(\alpha+t\beta)x+y\|}\lesssim_{\delta,\alpha,\beta} \log N$$
for each $x\in [0,1]\setminus E_{N,\delta}$. Thus, \eqref{equ:er5ahyso'} follows, even with  $N\log N$ instead of $N^\gamma$ in the denominator.

Let us now analyze the subcase $\alpha+t\beta=0$. Now,
$$\|\alpha (x+n)-\gamma_1+t\langle \beta (x+n)-\gamma_2\rangle\|=\|-\gamma_1-t\gamma_2+mt+\frac{\alpha}{\beta}[\beta (x+n)-\gamma_2]\|,$$
where $m$ is as before. Let $\xi$ be either $\gamma_1+t\gamma_2$ or $\gamma_1+t\gamma_2+t$. Let $$M:=(N\beta)^{\frac{(\gamma-1)(\gamma+\epsilon)}{\gamma+\epsilon-1}}D_\epsilon^{\frac{1}{\gamma+\epsilon-1}}.$$ From Lemma \ref{l7hrstmdncvjrhufherw;'q,.c} we have that for each $x\in [0,1)$
\begin{equation}
\label{jfhury40rlvbk5yih9}
\frac{1}{N^{\gamma}}\sum_{n=0\atop{\|\frac{\alpha}{\beta}[\beta (x+n)-\gamma_2]-\xi\|\ge \frac{1}{M}}}^{N-1}\frac1{\|\frac{\alpha}{\beta}[\beta (x+n)-\gamma_2]-\xi\|}\lesssim_\beta \frac{1}{N^\gamma}\sum_{n=0\atop{\|\frac{\alpha}{\beta}n-\xi\|\ge \frac{1}{M}}}^{N\beta}\frac1{\|\frac{\alpha}{\beta}n-\xi\|}\lesssim 1.
\end{equation}

Let $S(\xi)$ be the set of those $0\le n\le N-1$ such that $\|\frac{\alpha}{\beta}n-\xi\|\le \frac{1}{M}$. It was proved in Lemma \ref{l7hrstmdncvjrhufherw;'q,.c} that $S(\xi)$ has at most $N(\frac 1{MD_\epsilon})^{\frac1{\gamma+\epsilon}}$ points. This is  $O_{\alpha,\beta}(N^{1/10})$, since $\frac{\epsilon}{\gamma+\epsilon-1}\le \frac1{10}$.
Define
$$E_{N,\delta}:=\{x\in[0,1]:\|\alpha(x+n)-\gamma_1\|\lesssim_{\alpha,\beta}  \delta N^{-\frac1{10}}\text{ for some }n\in S(\xi)\},$$ and note that $|E_{N,\delta}|\le\frac\delta2$.
 For $n\in S(\xi)$ we will use the alternative estimate
$$A_n(x)\ge \|\alpha (x+n)-\gamma_1\|^2.$$
Thus, if $x\notin E_{N,\delta}$, then  $A_n(x)\gtrsim \delta^{2}N^{-\frac15}$, and thus
\begin{equation}
\label{jfhury40rlvbk5yih9'}
\frac{1}{N^\gamma}\sum_{n=0\atop{n\in S(\xi)}}^{N-1}\frac{1}{A_n(x)}\lesssim_{\alpha,\beta}  \frac{1}{N^\gamma}\delta^{-2}N^{\frac15}N^{\frac1{10}}\lesssim_{\delta,\alpha,\beta} 1.
\end{equation}
The result now follows from \eqref{jfhury40rlvbk5yih9} and \eqref{jfhury40rlvbk5yih9'}.

Finally, assume that $\alpha/\beta$ is Liouville. Then the result follows right away from  \eqref{xcniery8u3ifhtgp[r0-9dvruyfg73t895ugijk'} and the fact that $A_n(x)\ge \|\alpha (x+n)-\gamma_1\|^2.$

\end{proof}
We now begin the final part of the argument of part (a) of Theorem \ref{TT2}.
Let now $\gamma$ satisfy \eqref{xcniery8u3ifhtgp[r0-9} and \eqref{xcniery8u3ifhtgp[r0-9ufts} if $\alpha/\beta$ is Diophantine, and let $\gamma=2$ if $\alpha/\beta$ is Liouville.

Let  $S\subset [0,1)$ be a set that satisfies \eqref{e4}-\eqref{e6}, where now  $C=0$ (thus minimal decay is assumed). Note that since $f$ is no longer assumed to be continuous, all we can guarantee about $S$ is that it has positive measure. This however is enough to guarantee -via a classical result-  that $1_S*1_{-S}(0)>0$. Since $1_S*1_{-S}$ is continuous, there must exist an interval $I$ centered at the origin and $\delta>0$, such that $1_S*1_{-S}(p)>\delta$ on $I$. This implies that for each $p\in -I$,
$$|\{x\in S:x+p\in S\}|>\delta.$$
It is automatic (by pigeonholing as before) that there exists a sequence of integers $N_k\to\infty$ and $D=D(\alpha,\beta,s)\in (0,\infty)$ such that

(i) $\{P_k:=\frac{N_k}{\beta}\}<|I|/2$,

(ii) $N_k^{\gamma}\|N_k\frac{\alpha}{\beta}\|\le D$

 Let $E_{[P_k],\delta}$ be the exceptional set guaranteed by Lemma \ref{l7'}. This set will depend on $k$, but this dependence will not be relevant.
Then the set
$S_k:=\{x\in S:x-\{P_k\}\in S\}\cap ([0,1]\setminus E_{[P_k],\delta})$ is nonempty for each $k\ge 1$. Pick some $x_k\in S_k$, and finish the argument exactly like in the proof of Theorem \ref{t:main}, by working on the orbits of $x_k$ and $x_k-\{P_k\}$. In particular, note that by the argument in Proposition \ref{P1},
$$|\prod_{n=0}^{[P_k]-1}P(x_k-P_k+n)|\lesssim |\prod_{n=0}^{[P_k]-1}P(x_k+n)|,$$
with an implicit constant depending only on the fundamental parameters.

We caution that some implicit constants will now also depend on $\gamma$, $\epsilon$ and $\delta$. This is tolerable, since these   only depend on the  fundamental parameters.

Part (b) of Theorem \ref{TT2} is much simpler. Indeed, if say, $\alpha$ is rational, then
$$\|\alpha(x+n)-\gamma_1\|\gtrsim_{\delta,\alpha} 1$$ for each $n\in\N$ and each $x$ outside some $E_\delta$ with measure $\le \delta$. Lemma \ref{L4} implies that $|p(\alpha(x+n),\beta(x+n))|\gtrsim_{\delta,\alpha} 1$ for each $x\notin E_\delta$. In particular, the estimate in Lemma \ref{l7} $$\frac{1}{N}\sum_{n=0}^{N-1}\frac{1}{|P(x+n)|}\lesssim_{\delta,C_0,C_1,C_2,\alpha,\beta} 1$$
holds trivially (this time with $\gamma=1$). The rest of the argument is the same as in the proof of Theorem \ref{t:main}. One would have to apply the inequality above for $N=N_k$ satisfying
$$\sup_kN_k\|N_k\frac{\alpha}{\beta}\|\lesssim 1.$$

\section{Proof of Theorem \ref{TT2'}}
\label{S4}
The proof of part (a) of Theorem \ref{TT2'} follows the same general pattern as the proof of Theorem \ref{TT2}, but it is significantly simpler. We briefly sketch the details for part (a). As before, given $s>0$, there exists a sequence $N_k\to\infty$ of positive integers satisfying

(i) $\{P_k:=\frac{N_k}{\beta}\}<s$,

(ii) $\sup_kN_k\log N_k\|N_k\frac{\alpha}{\beta}\|\lesssim 1$

We have the following analogue of Proposition \ref{P1}.
\begin{proposition}
\label{P1bjfjghrthuy457t585urhrgyfgety}
Let $(N_k)$ be a sequence such that (ii) above holds, and define $P_k:=\frac{N_k}{\beta}$. Let $P(x)=A+Be(\alpha x)$, $Q(x)=C+De(\beta x)$, $R(x)=\frac{P(x)}{Q(x)}$. Given $\delta>0$,  there exists $E_{k,\delta}\subset [0,1]$ such that $|E_{k,\delta}|\le \delta$ such that  for each $x,y$ satisfying $x\in [0,1)\setminus E_{k,\delta}$ and $x=y+P_k$, we have
$$|\prod_{n=0}^{[P_k]-1}R(y+n)|\lesssim_{A,B,C,D,\alpha,\beta,\delta} |\prod_{n=0}^{[P_k]-1}R(x+n)|.$$
\end{proposition}
\begin{proof}

Lemma \ref{l7} implies  that  there exists an exceptional set $E_{k,\delta}\subset [0,1]$ such that
\begin{equation}
\label{yweytcn90r8t5t[hlpl;'lh;.;h,nfgfhjghjjhhfgj}
\frac{1}{N_k}\sum_{n=0}^{N_k}\frac{1}{|P(x+n)|}+\frac{1}{N_k}\sum_{n=0}^{N_k}\frac{1}{|Q(x+n)|}\lesssim_{\delta,\alpha,\beta,A,B,C,D}\log N_k,
\end{equation}
for each $x\in [0,1]\setminus E_{k,\delta}$. If $x=y+P_k$, we have
$$|e(\alpha x)-e(\alpha y)|=|e(P_k\alpha)-1|\lesssim \frac{1}{P_k\log P_k}$$
$$|e(\beta x)-e(\beta y)|=|e(P_k\beta)-1|\lesssim \frac{1}{P_k\log P_k}.$$
Thus, for each $n\in \N$
$$|P(y+n)|\le |P(x+n)|+O(\frac{1}{P_k\log P_k}),$$
$$|Q(y+n)|\ge |Q(x+n)|-O(\frac{1}{P_k\log P_k}).$$
Use the fact $a+b\le ae^{\frac{b}{a}}$ for each $a,b>0$ to get
$$|R(y+n)|\le |R(x+n)|e^{\frac{1}{P_k\log P_k |P(x+n)|}+\frac{1}{P_k\log P_k |Q(x+n)|}}.$$
The result now follows from \eqref{yweytcn90r8t5t[hlpl;'lh;.;h,nfgfhjghjjhhfgj}.
\end{proof}
The rest of the argument for part (a) is like in the proof of Theorem \ref{TT2}.

The almost periodicity argument is ineffective for  part (b) of Theorem \ref{TT2'}, because of \eqref{yweytcn90r8t5t[hlpl;'lh;.;h,n}. Our argument combines instead  the conjugates trick from the proof Theorem \ref{t:main} for (2,2) configurations, with a periodicity argument.
Assume there is a measurable function  $f:\R\to \C$, some $d\in (0,\infty)$ and some $S\subset [0,1]$ with positive measure such that
\begin{equation}
\label{eeeeeeeeeeeeeeeeeeeeeee2}
d^{-1}<|f(x)|<d\;\;\text{for almost every  }x\in S,
\end{equation}
and
\begin{equation}
\label{eeeeeeeeeeeeeeeeeeeeeee3}
\lim_{|n|\to\infty\atop_{n\in\Z}}f(x+n)=0,
\end{equation}
$$f(x+1)(A+Be(\alpha x))=f(x)(C+De(\beta x)),$$
for a.e. $x$, for some fixed $A,B,C,D\in\C$, $\alpha,\beta\in\R$, none of them zero. Let as before
$$P(x)=A+Be(\alpha x),\;\;\;Q(x)=C+De(\beta x).$$Assume also that $\beta=\frac{p}{r}$ is rational.
We can trivially assume $A$ is real and (by removing a countable set if necessary) that $S+\Z$ contains no zeros of $P$ and $Q$.

Let $B=re(\theta)$, with $r>0$. By invoking Birkhoff's pointwise ergodic theorem for the function $1_S$, there exists $x\in S$ and some $n'\in \N$ large enough such that $z:=\{-x-\frac{2\theta}{\alpha}+n'\alpha^{-1}\}\in S$ and such that $m:=[-x-\frac{2\theta}{\alpha}+n'\alpha^{-1}]>0$. Let $y:=-x-\frac{2\theta}{\alpha}+n'\alpha^{-1}$. We point out that $x,z,m$ are fixed for the rest of the argument, and thus they can be thought of as being fundamental parameters. The point of this selection is that for each $n\in\Z$,
$A+Be(\alpha y-n\alpha )$ and $A+Be(\alpha x+n\alpha)$ are complex conjugates
and thus, for each $N$
$$\prod_{n=-N+m}^{m}|P(z+n)|=\prod_{n=0}^{N}|P(x+n)|.$$
It follows that
\begin{equation}
\label{eeeeeeeeeeeeeeeeeeeeeee1}
\prod_{-N+m}^{n=-1}|P(z+n)|\gtrsim_{m,A,B}\prod_{n=0}^{N}|P(x+n)|
\end{equation}
Define $T(x)=|Q(x)Q(x+1)\ldots Q(x+r-1)|$. We distinguish two cases.

If $|T(x)|\ge |T(z)|$, then since $Q$ is $r$- periodic and $\inf_{k\in \Z}|Q(z+k)|>0$, we have
\begin{equation}
\label{eeeeeeeeeeeeeeeeeeeeeee4}
\prod_{n=-N+m}^{-1}|Q(z+n)|\lesssim_{m,x,z,C,D} \prod_{n=0}^{N}|Q(x+n)|.
\end{equation}
Using the fact that for each $N> m$
$$|f(x+N+1)|=|f(x)|\frac{\prod_{n=0}^{N}|Q(x+n)|}{\prod_{n=0}^{N}|P(x+n)|}$$
$$|f(z-N+m)|=|f(z)|\frac{\prod_{-N+m}^{n=-1}|P(z+n)|}{\prod_{-N+m}^{n=-1}|Q(z+n)|},$$
it follows that \eqref{eeeeeeeeeeeeeeeeeeeeeee2}-\eqref{eeeeeeeeeeeeeeeeeeeeeee4} can not hold simultaneously (just let $N\to\infty$).

If $|T(x)|\le |T(z)|$, then since $Q$ is $r$- periodic and $\inf_{k\in \Z}|Q(z+k)|>0$, we have

$$\prod_{n=-N+m}^{-1}|Q(x+n)|\lesssim_{m,C,D,x,z} \prod_{n=0}^{N}|Q(z+n)|.$$
Note also that as before,
$$\prod_{-N+m}^{n=-1}|P(x+n)|\gtrsim_{m,A,B,x,z}\prod_{n=0}^{N}|P(z+n)|.$$
Using the fact that for each $N\ge m$
$$|f(z+N+1)|=|f(z)|\frac{\prod_{n=0}^{N}|Q(z+n)|}{\prod_{n=0}^{N}|P(z+n)|}$$
$$|f(x-N+m)|=|f(x)|\frac{\prod_{-N+m}^{n=-1}|P(x+n)|}{\prod_{-N+m}^{n=-1}|Q(x+n)|},$$
the contradiction is forced like in the previous case.

\section{Open questions}
The argument from Theorem \ref{t:main}  seems to be too weak to tackle $(1,4)$ configurations like $(0,0),(1,0), (1,\alpha), (1,\beta), (1,\gamma)$. This is because the best one can guarantee in general is the existence of arbitrarily large $P$ such that $\max\{\|P\alpha\|,\|P\beta\|,\|P\gamma\|\}\lesssim \frac{1}{\sqrt{P}}$. It is not clear whether working with 3 or more orbits would have more to say about this case.

One can wonder if continuity can be removed from the proof of Theorem \ref{t:main} for special (2,2) configurations. If yes, Conjecture \ref{cc1} would follow right away for arbitrary (2,2) configurations. If no continuity is assumed, then the following is a typical worst case scenario
\begin{equation}
\label{eeeggtjjjjslllllllffffffff}
f(x+1)(1+e(\alpha x))=f(x)(1+e(\beta x)),
\end{equation}
with $\alpha,\beta$ distinct irrationals. The almost periodicity argument is ineffective in this case. Indeed, it is easy to see that
\begin{equation}
\label{yweytcn90r8t5t[hlpl;'lh;.;h,n}
\lim_{N\to\infty}\inf_{x\in [0,1]}\frac{1}{N}\sum_{n=1}^N\frac{1}{|1-e(x+\alpha n)|}=\infty.
\end{equation}
To see this, recall that by Dirichlet's theorem, for each $N$ there exists $p_N\in\N$ and $1\le m_N\le N$ relatively prime, such that
$$|\alpha-\frac{p_N}{m_N}|\le \frac{1}{m_NN}.$$
For each $i\in\{0,1,\ldots,m_N-1\}$, roughly $N/m_N$ of the points $1\le n\le N$ satisfy
$$\|\alpha n-\frac{i}{m_N}\|\le \frac{1}{m_N}.$$
Thus, for each $x\in [0,1]$
$$\frac{1}{N}\sum_{n=1}^N\frac{1}{|1-e(x+\alpha n)|}\gtrsim \sum_{i=1}^{m_N}\frac{1}{i}\sim \log m_N.$$
Finally, note that the irrationality of $\alpha$ forces $m_N\to\infty$.

It would be very interesting to know if given $\alpha\not=\beta$ the recurrence \eqref{eeeggtjjjjslllllllffffffff} can even hold along the orbit of a single point $x_0\in [0,1]$, in such a way that $\lim_{|N|\to\infty}f(x_0+N)=0$.

\section*{Acknowledgements}
We would like to thank Christoph Thiele  for stimulating our initial interest on the HRT conjecture. Special thanks to Zubin Gautam, Nets Katz, Joseph Rosenblatt and Alexandru Zaharescu for helpful discussions on the subject. The author is also indebted to Kasso Okoudjou for spotting a few inaccuracies in the proof from Section 3.

\end{document}